\documentclass[a4paper,10pt]{article}
\usepackage[frenchb]{babel}
\usepackage[T1]{fontenc}
\usepackage[utf8]{inputenc}
\usepackage{amsmath}
\usepackage{array}
\usepackage{amsthm}
\usepackage{amssymb}
\input xy
\xyoption{all}

\newcommand{\A}{{\mathcal{A}}}

\newcommand{\C}{{\mathcal{C}}}

\newcommand{\F}{{\mathcal{F}}}

\newcommand{\I}{{\mathcal{I}}}

\newcommand{\T}{{\mathcal{T}}}

\newcommand{\ti}{{\text{-}}}

\title{Cohomologie des foncteurs polynomiaux sur les groupes libres}

\author{Aur\'elien DJAMENT\thanks{CNRS, Laboratoire de mathématiques Jean Leray (UMR 6629), aurelien.djament@univ-nantes.fr.}\,, Teimuraz PIRASHVILI\thanks{Department of Mathematics, University of Leicester, University Road, Leicester, LE1
7RH, UK, tp59@leicester.ac.uk. Cet auteur est partiellement soutenu par la bourse ST08/3-387 de la GNSF.}\, et Christine VESPA\thanks{Institut de Recherche Mathématique Avancée, université de Strasbourg,
vespa@math.unistra.fr. Cet auteur est partiellement soutenu par le projet ANR-11-BS01-0002
HOGT : Homotopie, Opérades et Groupes de Grothendieck-Teichmüller.}}

\newtheorem{thi}{Th\'eor\`eme}

\newtheorem{thm}{Th\'eor\`eme}[section]
\newtheorem{pr}[thm]{Proposition}
\newtheorem{cor}[thm]{Corollaire}
\newtheorem{lm}[thm]{Lemme}

\theoremstyle{definition}

\newtheorem{nota}[thm]{Notation}

\theoremstyle{remark}
\newtheorem{rem}[thm]{Remarque}

\begin{document}

\maketitle

\begin{abstract}
On montre que les groupes d'extensions entre foncteurs polynomiaux sur les groupes libres sont les mêmes dans la catégorie de tous les foncteurs et dans une sous-catégorie de foncteurs polynomiaux de degré borné. On donne quelques applications. 
\end{abstract}

\begin{small}
\begin{center}
 \textbf{Abstract}
\end{center}

\smallskip

We show that extension groups between two polynomial functors on free groups are the same in the category of all functors and in a subcategory of polynomial functors of bounded degree. We give some applications.

\end{small}

\medskip

\noindent
{\em Mots clefs} : catégories de foncteurs ; groupes libres ; foncteurs polynomiaux ; groupes d'extensions.

\smallskip

\noindent
{\em Classification MSC 2010} :  18A25, 18G15, 20J15 (18E15, 18G10).

\section*{Introduction}

Ce travail est une contribution à l'étude des groupes d'extensions dans la catégorie des foncteurs des groupes libres (de rang fini) vers les groupes abéliens. Cette catégorie a d'abord été considérée, du point de vue cohomologique, dans l'article de Jibladze et du deuxième auteur \cite{PJ} (§\,5A) sur la cohomologie des théories algébriques. Plus récemment, les premier et troisième auteurs l'ont utilisée, dans \cite{DV2}, pour établir des résultats d'annulation d'homologie stable des groupes d'automorphismes des groupes libres à coefficients tordus. Conjecturalement, l'algèbre homologique dans cette catégorie de foncteurs devrait gouverner le calcul d'autres groupes d'homologie stable des groupes d'automorphismes des groupes libres à coefficients tordus. Cette algèbre homologique s'avère plus facile d'accès que dans les catégories de foncteurs entre espaces vectoriels sur un corps fini, très étudiées en raison de leurs liens avec la topologie algébrique ou la $K$-théorie algébrique (cf. par exemple \cite{FFPS}), notamment parce que le foncteur d'abélianisation possède une résolution projective explicite simple (déjà utilisée dans \cite{PJ}). L'article \cite{DV2} se sert également de façon cruciale de la structure des sous-catégories de {\em foncteurs polynomiaux} dans cette catégorie de foncteurs.

  Dans le présent travail, on montre que les groupes d'extensions\,\footnote{Un résultat similaire vaut pour les groupes de torsion.} entre deux foncteurs polynomiaux sur les groupes libres sont les mêmes dans la catégorie de tous les foncteurs ou dans une sous-catégorie de foncteurs polynomiaux de degré donné. Ce résultat contraste avec la situation, plus délicate, des foncteurs sur une catégorie additive (voir \cite{P-extor} et \cite{Dja-Pg}). Nous donnons également quelques applications.

\paragraph*{Description des résultats}

Commençons par quelques notations générales. Si $\C$ est une petite catégorie, on note $\F(\C)$ la catégorie des foncteurs de $\C$ vers la catégorie $\mathbf{Ab}$ des groupes abéliens. Si la catégorie $\C$ possède un objet nul $0$ et est munie d'une structure monoïdale symétrique dont $0$ est l'unité --- par exemple, le coproduit catégorique, lorsqu'il existe dans $\C$ ---, on dispose d'une notion d'{\em effets croisés} et de {\em foncteurs polynomiaux} dans $\F(\C)$. L'origine de ces notions remonte à Eilenberg-MacLane (\cite{EML}, {\em Chapter}~II), lorsque $\C$ est une catégorie de modules, et se généralise sans difficulté au cas qu'on vient de mentionner (voir \cite{HPV}, §\,2). Pour tout entier $d$, on note $\F_d(\C)$ la sous-catégorie pleine de $\F(\C)$ des foncteurs polynomiaux de degré au plus $d$.

On note $\mathbf{gr}$ la catégorie des groupes libres de rang fini, ou plus exactement le squelette constitué des groupes $\mathbb{Z}^{*n}$ (l'étoile désignant le produit libre). Comme signalé plus haut, notre résultat principal est le suivant :

\begin{thi}\label{thf}
 Soient $d\in\mathbb{N}$ et $F$, $G$ des objets de $\F_d(\mathbf{gr})$. L'application linéaire graduée naturelle
$${\rm Ext}^*_{\F_d(\mathbf{gr})}(F,G)\to {\rm Ext}^*_{\F(\mathbf{gr})}(F,G)$$
qu'induit le foncteur d'inclusion $\F_d(\mathbf{gr})\to\F(\mathbf{gr})$ est un isomorphisme.
\end{thi}

La démonstration de ce théorème est donnée à la section~\ref{sd}. Elle repose sur une propriété d'annulation cohomologique très inspirée d'une propriété analogue dans les catégories de foncteurs sur une catégorie additive due au deuxième auteur (voir \cite{P-add}) et sur les propriétés de l'anneau gradué associé à la filtration de l'anneau d'un produit direct de groupes libres par les puissances de son idéal d'augmentation, propriétés qui sont rappelées dans la section~\ref{sp}. La dernière section est consacrée aux applications. Tout d'abord, on calcule la dimension homologique des puissances tensorielles de l'abélianisation et des foncteurs de Passi dans les catégories $\F_d(\mathbf{gr})$ (proposition~\ref{pr-dh} et corollaire~\ref{cor-passi}). On montre ensuite que les foncteurs polynomiaux de degré donné de $\mathbf{gr}$ vers les espaces vectoriels sur $\mathbb{Q}$ forment des catégories de dimension globale finie (proposition~\ref{prdq}). Nous montrons également, dans la proposition~\ref{pr-beta}, qu'un foncteur adjoint apparemment mystérieux (et en tout cas non explicite) mentionné dans \cite{DV2} possède un comportement cohomologique agréable. 

\paragraph*{Quelques notations et rappels}

Si $k$ est un anneau, on note $k\ti\mathbf{Mod}$ la catégorie des $k$-modules à gauche. On note aussi $\F(\C;k)$, lorsque $\C$ est une petite catégorie, la catégorie des foncteurs de $\C$ vers $k\ti\mathbf{Mod}$. Ainsi, $\F(\C)=\F(\C;\mathbb{Z})$. La catégorie $\F(\C;k)$ est une catégorie abélienne qui hérite des propriétés de régularité de la catégorie but $k\ti\mathbf{Mod}$ ; en effet, les suites exactes, sommes, produits, etc. se testent au but. Si $k$ est commutatif, on dispose d'une structure monoïdale symétrique notée $\otimes$ sur $\F(\C;k)$ qui est le produit tensoriel sur $k$ au but.

La catégorie $\F(\C;k)$ possède suffisamment d'objets projectifs : en effet, le lemme de Yoneda montre que le foncteur
$$P^\C_c:=k[\C(c,-)]$$
(où $c$ est un objet de $\C$ ; $k[-]$ désigne le foncteur de linéarisation des ensembles vers $k\ti\mathbf{Mod}$) représente l'évaluation en $c$, il est donc projectif (et de type fini), et les foncteurs $P^\C_c$ engendrent la catégorie $\F(\C;k)$.

Pour alléger, on notera $P^\mathbf{gr}_n$ pour $P^\mathbf{gr}_{\mathbb{Z}^{*n}}$. Comme $*$ est un coproduit dans la catégorie $\mathbf{gr}$, on dispose d'isomorphismes canoniques $P^\mathbf{gr}_i\otimes P^\mathbf{gr}_j\simeq P^\mathbf{gr}_{i+j}$.

Si la catégorie $\C$ possède un objet nul $0$ et est munie d'une structure monoïdale symétrique dont $0$ est l'unité, on dispose dans $\F(\C;k)$ comme dans $\F(\C)$ d'une notion d'effets croisés et de foncteurs polynomiaux sur $\F(\C;k)$ ; on note encore $\F_n(\C;k)$ la sous-catégorie pleine de $\F(\C;k)$ constituée des foncteurs polynomiaux de degré au plus $n$. C'est une sous-catégorie épaisse de $\F(\C;k)$ stable par limites et colimites. On rappelle également que le produit tensoriel de deux foncteurs polynomiaux est polynomial (on suppose ici $k$ commutatif), avec pour degré la somme des degrés des foncteurs initiaux. Le lecteur pourra trouver davantage de détails et de propriétés des foncteurs polynomiaux dans \cite{HPV}, §\,2.

Comme la catégorie $\mathbf{gr}$ est pointée, tout foncteur de $\F(\mathbf{gr})$ se scinde de manière unique (à isomorphisme près) et naturelle en la somme directe d'un foncteur constant et d'un foncteur {\em réduit}, c'est-à-dire nul sur le groupe trivial. On notera $\bar{P}$ la partie réduite de $P^\mathbf{gr}_1$.

\paragraph*{Remerciements} Les deux derniers auteurs souhaitent remercier les organisateurs du semestre intitulé {\em Grothendieck-Teichmuller Groups, Deformation and Operads} qui a eu lieu à l'Institut Isaac Newton de Cambridge en 2013 et pendant lequel ce travail à débuté ainsi que l'Institut pour son hospitalité et les excellentes conditions de travail qui y sont proposées. Le second auteur est reconnaissant envers l'université de Leicester pour lui avoir accordé un congé pour recherche.

\section{Préliminaires sur les puissances de l'idéal d'augmentation d'un anneau de groupe}\label{sp}

(Les propriétés rappelées dans cette section sont classiques. Une référence générale est \cite{Passi}.)

Pour tout groupe $G$, on note $\I(G)$ l'idéal d'augmentation de l'anneau de groupe $\mathbb{Z}[G]$ ; pour tout $n\in\mathbb{N}$, on note $\I^n(G)$ la $n$-ème puissance de cet idéal. Pour tout $n\in\mathbb{N}$, $\I^n$ définit un foncteur de la catégorie $\mathbf{Grp}$ des groupes vers la catégorie $\mathbf{Ab}$. On dispose ainsi d'un anneau gradué ${\rm gr}(\mathbb{Z}[G]):=\underset{n\in\mathbb{N}}{\bigoplus}(\I^n/\I^{n+1})(G)$ naturel en $G$. En degré $1$, on dispose d'un isomorphisme naturel $(\I/\I^2)(G)\simeq G_{ab}$ (abélianisation de $G$). En particulier, le produit induit des morphismes naturels $G_{ab}^{\otimes n}\to (\I^n/\I^{n+1})(G)$, qui sont toujours des épimorphismes. Ainsi, on dispose d'un morphisme naturel surjectif d'anneaux gradués de l'algèbre tensorielle sur $G_{ab}$ vers ${\rm gr}(\mathbb{Z}[G])$. Une conséquence importante de la théorie de Magnus est que ce morphisme est un isomorphisme si $G$ est un groupe libre (voir par exemple \cite{Passi}, chapitre~VIII, {\em Theorem}~6.2).

\begin{lm}\label{lmi1}
 Si $G$ est un groupe libre, pour tous entiers naturels $r$ et $i$, on dispose d'un isomorphisme naturel
$$H_i(G;\I^r(G))\simeq\left\lbrace\begin{array}{ll}
 (\I^r/\I^{r+1})(G) & \text{si } i=0\\
 0 & \text{sinon.}
 \end{array}
 \right.$$
\end{lm}

\begin{proof}
 On procède par récurrence sur $r$. L'assertion est triviale pour $r=0$, on suppose donc $r>0$ et le résultat établi pour $r-1$. Dans la suite exacte de $\mathbb{Z}[G]$-modules à gauche
$$0\to\I^r(G)\to\I^{r-1}(G)\to (\I^{r-1}/\I^r)(G)\to 0,$$
l'action de $G$ est {\em triviale} sur $(\I^{r-1}/\I^r)(G)$, car la flèche de droite s'identifie à la projection de $\I^{r-1}(G)$ sur ses coïnvariants sous l'action de $G$ --- ce qui prouve déjà le résultat en degré homologique $0$. Comme l'homologie de $G$ est sans torsion sur $\mathbb{Z}$ et concentrée en degré $0$ et $1$, puisque $G$ est libre, on en déduit que $H_i(G;(\I^{r-1}/\I^r)(G))$ est nul pour $i>1$, isomorphe à $G_{ab}\otimes (\I^{r-1}/\I^r)(G)$ pour $i=1$ et à $(\I^{r-1}/\I^r)(G)$ pour $i=0$. La suite exacte longue d'homologie associée à la suite exacte courte précédente fournit donc le résultat d'annulation souhaité en degré homologique strictement positif. 
\end{proof}

\begin{rem}
 Cette démonstration permet aussi de voir, en regardant la fin de ladite suite exacte longue, que le morphisme de liaison
$$G_{ab}\otimes (\I^{r-1}/\I^r)(G)\simeq H_1(G;(\I^{r-1}/\I^r)(G))\to H_0(G;\I^r(G))\simeq (\I^r/\I^{r+1})(G)$$
est un isomorphisme. Comme ce morphisme s'identifie au produit
$$(\I/\I^2)(G)\otimes (\I^{r-1}/\I^r)(G)\to (\I^r/\I^{r+1})(G),$$
cela fournit une démonstration du fait classique rappelé plus haut que l'anneau gradué ${\rm gr}(\mathbb{Z}[G])$ est naturellement isomorphe à l'algèbre tensorielle sur $G_{ab}$ (pour $G$ libre).
\end{rem}

\begin{lm}\label{lmtgl}
 Soient $G$ un groupe libre, $r$ et $i$ des entiers naturels. On dispose d'un isomorphisme naturel
$${\rm Tor}^{\mathbb{Z}[G]}_i(\I(G),\I^r(G))\simeq\left\lbrace\begin{array}{ll}
 \I^{r+1}(G) & \text{si } i=0\\
 0 & \text{sinon.}
 \end{array}
 \right.$$
\end{lm}

\begin{proof}
 Utilisant la suite exacte longue d'homologie associée à la suite exacte courte de $\mathbb{Z}[G]$-modules à droite $0\to\I(G)\to\mathbb{Z}[G]\to\mathbb{Z}\to 0$, on voit que ${\rm Tor}^{\mathbb{Z}[G]}_i(\I(G),\I^r(G))$ est isomorphe à $H_{i+1}(G;\I^r(G))$ lorsque $i>0$, et l'on obtient une suite exacte
$$0\to H_1(G;\I^r(G))\to {\rm Tor}^{\mathbb{Z}[G]}_0(\I(G),\I^r(G))\to\I^r(G)\to H_0(G;\I^r(G))\to 0,$$
de sorte que le lemme~\ref{lmi1} permet de conclure.
\end{proof}

\begin{pr}\label{pr-seb}
Pour tout $r\in\mathbb{N}$, il existe un complexe de chaînes de foncteurs $\mathbf{Grp}\to\mathbf{Ab}$ du type
$$\dots\to\I^{\otimes (n+1)}\otimes\I^r\to\I^{\otimes n}\otimes\I^r\to\dots\to\I^{\otimes 2}\otimes\I^r\to\I\otimes\I^r$$
dont la restriction aux groupes libres a une homologie isomorphe à $\I^{r+1}$ en degré $0$ et nulle en degré strictement positif.
\end{pr}

\begin{proof}
 D'une manière générale, si $A$ est un anneau augmenté, d'idéal d'augmentation $\bar{A}$, $M$ un $A$-module à droite et $N$ un $A$-module à gauche, on dispose d'un complexe de chaînes de groupes abéliens fonctoriel en $A$, $M$ et $N$
$$\dots\to M\otimes\bar{A}^{\otimes n}\otimes N\to M\otimes\bar{A}^{\otimes (n-1)}\otimes N\to\dots\to M\otimes\bar{A}\otimes N\to M\otimes N$$
(complexe de Hochschild normalisé) dont l'homologie est naturellement isomorphe à l'homologie de Hochschild de $A$ à coefficients dans le bimodule $M\otimes N$. Celle-ci est naturellement isomorphe à ${\rm Tor}^A_*(M,N)$ si $A$ et $M$ sont sans torsion sur $\mathbb{Z}$. (On pourra se référer par exemple à \cite{Wei}, §\,9.1 pour ces propriétés élémentaires de l'homologie de Hochschild.)

La proposition s'obtient en prenant $A=\mathbb{Z}[G]$, $M=\I(G)$ et $N=\I^r(G)$ et en appliquant le lemme~\ref{lmtgl}.
\end{proof}

\smallskip

Nous aurons besoin également d'examiner ${\rm gr}(\mathbb{Z}[G])$ lorsque $G$ est un produit direct de groupes libres. Pour cela, on donne quelques propriétés générales simples sur l'effet du foncteur ${\rm gr}(\mathbb{Z}[-])$ sur un produit direct de groupes.

On va voir que ce foncteur est {\em exponentiel} (i.e. transforme les produits directs en produits tensoriels) sur les groupes $G$ tels que ${\rm gr}(\mathbb{Z}[G])$ soit un groupe abélien sans torsion. (On a vu plus haut que les groupes libres possèdent cette propriété.) Remarquer que, si c'est le cas, alors les groupes abéliens $\I^n(G)$ et $\mathbb{Z}[G]/\I^n(G)$ sont également sans torsion. 

Soient $G$ et $H$ deux groupes. Si $M$ et $N$ sont des sous-groupes des groupes abéliens $\mathbb{Z}[G]$ et $\mathbb{Z}[H]$ respectivement, nous noterons $M.N$ l'image de $M\otimes N$ dans $\mathbb{Z}[G]\otimes\mathbb{Z}[H]\simeq\mathbb{Z}[G\times H]$. Si $M$ et $N$ sont des idéaux bilatères de $\mathbb{Z}[G]$ et $\mathbb{Z}[H]$ respectivement, alors $M.N$ est un idéal bilatère de $\mathbb{Z}[G\times H]$.

\begin{pr}\label{prp1}
 Soient $G$ et $H$ deux groupes. Dans $\mathbb{Z}[G\times H]$, on a
$$\I(G\times H)=\I(G).\mathbb{Z}[H]+\mathbb{Z}[G].\I(H)$$
et, plus généralement,
$$\I^r(G\times H)=\underset{i+j=r}{\sum}\I^i(G).\I^j(H)$$
pour tout $r\in\mathbb{N}$.
\end{pr}

\begin{proof}
La première propriété est immédiate et la deuxième s'en déduit par récurrence sur $r$. 
\end{proof}

\begin{pr}\label{prgp}
 Soient $G$ et $H$ deux groupes tels que les groupes abéliens ${\rm gr}(\mathbb{Z}[G])$ et ${\rm gr}(\mathbb{Z}[H])$ soient sans torsion.
\begin{enumerate}
 \item Pour tout entier naturel $n$, on dispose d'un isomorphisme naturel
$$(\I^n/\I^{n+1})(G\times H)\simeq\underset{i+j=n}{\bigoplus}(\I^i/\I^{i+1})(G)\otimes (\I^j/\I^{j+1})(H).$$
\item Le groupe abélien ${\rm gr}(\mathbb{Z}[G\times H])$ est sans torsion.
\item Pour tous entiers $0\leq t\leq n$, on a
$$\Big(\underset{i<t}{\underset{i+j=n}{\sum}}\I^i(G).\I^j(H)\Big)\cap (\I^t(G).\I^{n-t}(H))=\I^t(G).\I^{n-t+1}(H)$$
dans $\mathbb{Z}[G\times H]$ ; ce groupe est naturellement isomorphe à $\I^t(G)\otimes\I^{n-t+1}(H)$.
\end{enumerate}
\end{pr}

\begin{proof}
 La proposition~\ref{prp1} procure (sans aucune hypothèse sur $G$ ni $H$) un épimorphisme naturel
$$\underset{i+j=n}{\bigoplus}(\I^i/\I^{i+1})(G)\otimes (\I^j/\I^{j+1})(H)\twoheadrightarrow (\I^n/\I^{n+1})(G\times H).$$

Par ailleurs, grâce à l'hypothèse faite sur $G$ et $H$, les épimorphismes ca\-no\-niques $\I^i(G)\otimes\I^j(H)\twoheadrightarrow \I^i(G).\I^j(H)$ sont des isomorphismes, et l'on a
$$\I^i(G).\I^{n-i}(H)\cap \I^j(G).\I^{n-j}(H)=\I^j(G).\I^{n-i}(H)$$
pour $i\leq j$. On en déduit immédiatement la proposition.
\end{proof}

\section{Démonstration du théorème~\ref{thf}}\label{sd}

\paragraph*{Les classes $\T_n$ et $\T'_n$}

Pour tout $n\in\mathbb{N}$, notons $\T_n$ la classe des objets $F$ de $\F(\mathbf{gr})$ tels que ${\rm Ext}^*_{\F(\mathbf{gr})}(F,A)=0$ pour tout objet $A$ de $\F_n(\mathbf{gr})$. Nous utiliserons également la classe $\T'_n$ des objets $F$ de $\F(\mathbf{gr})$ possédant une résolution projective dont les termes sont des sommes directes de foncteurs du type $\bar{P}^{\otimes d}$ avec $d>n$. (On rappelle que $\bar{P}$ désigne la partie réduite du projectif $P^\mathbf{gr}_\mathbb{Z}$ ; ce foncteur s'identifie à la restriction à $\mathbf{gr}$ de $\I$.)

 Ainsi, $\T_n\supset\T_{n+1}$ et $\T'_n\supset\T_{n+1}$ pour tout $n\in\mathbb{N}$, et $\T_0=\T'_0$ est constituée des foncteurs réduits. 

La propriété suivante est immédiate.

\begin{lm}\label{lmet}
 Soient $n\in\mathbb{N}$ et $0\to F\to G\to H\to 0$ une suite exacte courte de $\F(\mathbf{gr})$. Si deux des foncteurs $F$, $G$ et $H$ appartiennent à $\T_n$ (resp. $\T'_n$), alors il en est de même pour le troisième.

En particulier, si $A$ et $B$ sont deux sous-foncteurs d'un foncteur $F$ de $\F(\mathbf{gr})$, $A$ et $B$ appartenant à $\T_n$ (resp. $\T'_n$), alors le sous-foncteur $A+B$ de $F$ appartient à $\T_n$ (resp. $\T'_n$) si et seulement s'il en est de même pour $A\cap B$.

Plus généralement, si on dispose d'une suite exacte longue
$$\dots\to X_1\to X_0\to F\to 0$$
dans $\F(\mathbf{gr})$ avec tous les $X_i$ dans $\T_n$ (resp. $\T'_n$), alors $F$ appartient à $\T_n$ (resp. $\T'_n$).
\end{lm}

Les classes $\T'_n$ nous serviront par l'intermédiaire du résultat suivant.

\begin{pr}\label{pr-pira}
 Pour tout entier $n\in\mathbb{N}$, la classe $\T'_n$ est incluse dans $\T_n$.
\end{pr}

\begin{proof}
 Il suffit de montrer que ${\rm Hom}_{\F(\mathbf{gr})}(\bar{P}^{\otimes d},F)=0$ lorsque $F$ est polynomial de degré strictement inférieur à $d$. Cela découle de ce que le facteur direct ${\rm Hom}_{\F(\mathbf{gr})}(\bar{P}^{\otimes d},F)$ de ${\rm Hom}_{\F(\mathbf{gr})}(P^\mathbf{gr}_d,F)\simeq F(\mathbb{Z}^{*d})$ s'identifie à l'effet croisé $cr_d(F)(\mathbb{Z},\dots,\mathbb{Z})$.
\end{proof}


L'intérêt des classes $\T'_n$ réside dans leur comportement agréable relativement aux produits tensoriels.

\begin{pr}\label{prpp}
 Si $F$ et $G$ appartiennent à $\T'_i$ et $\T'_j$ respectivement et que l'un de ces foncteurs prend des valeurs sans torsion sur $\mathbb{Z}$, alors $F\otimes G$ appartient à $\T'_{i+j+1}$.
\end{pr}

\begin{proof}
 Cela découle de la formule de Künneth.
\end{proof}

\begin{cor}\label{corpt}
 Le produit tensoriel de $n+1$ foncteurs réduits de $\F(\mathbf{gr})$ dont au moins $n$ prennent des valeurs sans torsion sur $\mathbb{Z}$ appartient à $\T'_n$.
\end{cor}

\begin{rem}
Ce corollaire (dans lequel on pourrait remplacer, comme dans ce qui précède, $\mathbf{gr}$ par n'importe quelle petite catégorie pointée avec coproduits finis) est un analogue d'un résultat dû au deuxième auteur pour les catégories $\F(\A)$, où $\A$ est additive, qui apparaît dans \cite{P-add}, et qui s'est avéré extrêmement utile en cohomologie des foncteurs.
\end{rem}

\paragraph*{Les foncteurs $K^d_n$}

Pour tout $d\in\mathbb{N}$, le foncteur d'inclusion $\F_d(\mathbf{gr})\to\F(\mathbf{gr})$ possède un adjoint à gauche $q_d$ ; $q_d(F)$ est le plus grand quotient de $F$ appartenant à $\F_d(\mathbf{gr})$. On renvoie à \cite{HPV}, §\,2.3 pour plus de détails à ce sujet.

\begin{nota}
 Étant donné deux entiers naturels $n$ et $d$, on pose
$$K^d_n={\rm Ker} (P^\mathbf{gr}_n\twoheadrightarrow q_d(P^\mathbf{gr}_n)).$$
\end{nota}

\begin{pr}\label{ida}
 Pour tous $n, d\in\mathbb{N}$ et tout objet $G$ de $\mathbf{gr}$, on dispose d'un isomorphisme naturel
$$K^d_n(G)\simeq\I^{d+1}(G^n).$$
\end{pr}

\begin{proof}
 Notons $J^d_n$ le sous-foncteur de $P^\mathbf{gr}_n$ donné par $G\mapsto \I^{d+1}(G^n)\subset\mathbb{Z}[G^n]\simeq P^\mathbf{gr}_n(G)$. Alors le foncteur $P^\mathbf{gr}_n/J^d_n$ appartient à $\F_d(\mathbf{gr})$. Cela résulte de ce que, pour tout $r\in\mathbb{N}$, le foncteur $G\mapsto (\I^r/\I^{r+1})(G^n)$ (des groupes vers les groupes abéliens) est polynomial de degré (au plus) $r$, puisque c'est un quotient de $G\mapsto (G_{ab}^n)^{\otimes r}$ (voir le début de la section~\ref{sp}). On a donc $K^d_n\subset J^d_n$.

Montrons l'inclusion inverse. Pour tout foncteur $F$ de $\F(\mathbf{gr})$ et tout objet $G$ de $\mathbf{gr}$, on considère l'application naturelle
$$\kappa_d(F)(G)=\underset{I\subset\{0,1,\dots,d\}}{\sum}(-1)^{{\rm Card}(I)}F(p_I) : F(G^{* (d+1)})\to F(G),$$
où $p_I\in\mathbf{gr}(G^{* (d+1)},G)\simeq {\rm End}(G)^{d+1}$ est le morphisme dont la $i$-ème composante est l'identité si $i\in I$ et le morphisme trivial sinon. La transformation naturelle $\kappa_d(F)$ est nulle lorsque $F$ appartient à $\F_d(\mathbf{gr})$, car elle se factorise par l'idempotent de $F(G^{* (d+1)})$ dont l'image est par définition l'effet croisé $cr_{d+1}(F)(G,\dots,G)$.

Par conséquent, comme $q_d(F)$ appartient à $\F_d(\mathbf{gr})$, la composée
$$F(G^{* (d+1)})\xrightarrow{\kappa_d(F)(G)}F(G)\twoheadrightarrow q_d(F)(G)$$
est nulle, de sorte que le noyau de $\kappa_d(F)$ contient\,\footnote{Il y a en fait égalité --- cf. \cite{HPV}, définition~2.16 et proposition~2.17 ---, mais nous n'en aurons pas besoin.} l'image de $\kappa_d$. En particulier, $K^n_d$ contient l'image de $\kappa_d(P^\mathbf{gr}_n)$. Si $(g^i_j)_{1\leq i\leq n, 0\leq j\leq d}$ est une famille d'éléments d'un groupe libre de rang fini $G$, on a
$$\kappa_d(P^\mathbf{gr}_n)(G)\big([(g^i_0*\dots*g^i_d)_{1\leq i\leq n}]\big)=\underset{0\leq j_1<\dots<j_r\leq d}{\sum}(-1)^{r}[(g^i_{j_1}\dots g^i_{j_r})_{1\leq i\leq n}]$$
qui est égal au produit
$$a_0 a_1\dots a_d\quad\text{où}\quad a_j:=[(g^i_j)_{1\leq i\leq n}]-[1]\in\I(G^n)\subset\mathbb{Z}[G^n].$$
Comme $\I^{d+1}(G^n)$ est engendré par ces produits, cela montre que $K^n_d$ contient $J^d_n$, d'où la conclusion.
\end{proof}

\begin{lm}\label{lmr1}
 Pour tout $d\in\mathbb{N}$, le foncteur $K^d_1$ appartient à $\T'_d$.
\end{lm}

\begin{proof}
 On procède par récurrence sur $d$. Pour $d=0$, c'est clair. 

Compte-tenu de la proposition~\ref{ida}, la proposition~\ref{pr-seb} peut se traduire par l'existence de suites exactes
$$\dots\to\bar{P}^{\otimes (n+1)}\otimes K^{d-1}_1\to\bar{P}^{\otimes n}\otimes K^{d-1}_1\to\dots\to\bar{P}^{\otimes 2}\otimes K^{d-1}_1\to\bar{P}\otimes K^{d-1}_1\to K^d_1\to 0$$
dans $\F(\mathbf{gr})$. Si $K^{d-1}_1$ appartient à $\T'_{d-1}$, alors $\bar{P}^{\otimes i}\otimes K^{d-1}_1$ appartient à $\T'_d$ pour tout entier $i\geq 1$, donc $K^d_1$ aussi (en utilisant le lemme~\ref{lmet}), d'où le lemme.
\end{proof}

\begin{pr}\label{prpf}
 Pour tous $d, n\in\mathbb{N}$, le foncteur $K^d_n$ appartient à $\T'_d$.
\end{pr}

\begin{proof}
 On montre par récurrence sur $n$ que $K^d_n$ appartient à $\T'_d$ pour tout $d$. Pour $n=0$ il n'y a rien à faire ; compte-tenu du lemme précédent, on peut supposer $n>1$ et l'assertion établie pour $n-1$.

Étant donné $d\in\mathbb{N}$, on observe que, pour tout $t\in\{1,\dots,d+2\}$, le foncteur
$$\underset{i<t}{\underset{i+j=d+1}{\sum}}K^{i-1}_{n-1}.K^{j-1}_1\subset P^\mathbf{gr}_n$$
(où l'on conserve les notations de la fin de la section~\ref{sp} ; on fait sans cesse usage de l'identification canonique $P^\mathbf{gr}_n(G)=\mathbb{Z}[G^n]$ et de la proposition~\ref{ida}) appartient à $\T'_d$. En effet, $K^a_{n-1}. K^b_1\simeq K^a_{n-1}\otimes K^b_1$ appartient à $\T'_{a+b+1}$ par la proposition~\ref{prpp} et l'hypothèse de récurrence. On obtient donc l'assertion par récurrence sur $t$ en utilisant la dernière partie de la proposition~\ref{prgp} et le lemme~\ref{lmet}. En prenant $t=d+2$ et en utilisant la proposition~\ref{prp1}, on obtient bien que $K^d_n$ appartient à $\T'_d$, d'où la conclusion.
\end{proof}

\begin{proof}[Démonstration du théorème~\ref{thf}] Comme les foncteurs $P^\mathbf{gr}_n$ forment un ensemble de générateurs projectifs de $\F(\mathbf{gr})$ lorsque $n$ décrit $\mathbb{N}$, pour tout $d\in\mathbb{N}$ fixé, les foncteurs $q_d(P^\mathbf{gr}_n)$ forment un ensemble de générateurs projectifs de $\F_d(\mathbf{gr})$ (car $q_d$ est adjoint à gauche au foncteur exact d'inclusion). Par conséquent, il suffit de montrer le théorème~\ref{thf} pour $F=q_d(P^\mathbf{gr}_n)$, auquel cas il équivaut à la nullité de ${\rm Ext}^*_{\F(\mathbf{gr})}(K^d_n,G)$ (puisque ${\rm Ext}^*_{\F(\mathbf{gr})}(P^\mathbf{gr}_n,G)$ et ${\rm Ext}^*_{\F_d(\mathbf{gr})}(q_d(P^\mathbf{gr}_n),G)$ sont nuls en degré cohomologique strictement positif), lorsque $G$ appartient à $\F_d(\mathbf{gr})$. Autrement dit, il s'agit de montrer que $K^d_n$ appartient à $\T_d$. Cela découle des propositions~\ref{prpf} et~\ref{pr-pira}.
\end{proof}

\begin{rem}
 \begin{enumerate}
  \item Signalons une méthode alternative (qui repose toutefois exactement sur les mêmes idées) pour démontrer le théorème. On remarque d'abord qu'il suffit de l'établir lorsque $F$ est une puissance tensorielle du foncteur d'abélianisation, grâce aux résultats de structure des foncteurs polynomiaux sur $\mathbf{gr}$ (cf. \cite{DV2} §\,2). Ces foncteurs possèdent une résolution projective explicite obtenue en prenant la puissance tensorielle appropriée de la résolution bar (décalée) --- voir (\ref{eqb}) ci-après. Le c\oe ur de l'approche alternative qu'on esquisse consiste à montrer qu'appliquer le foncteur $q_d$ à ces résolutions donne encore des complexes acycliques. Cela se fait par récurrence sur $d$ (en considérant le noyau de $q_d\twoheadrightarrow q_{d-1}$) et en utilisant les propriétés rappelées dans la section~\ref{sp} (ou des variantes).
\item Si $\C$ est une petite catégorie pointée avec coproduits finis, il est exceptionnel que l'inclusion de la sous-catégorie de foncteurs polynomiaux $\F_d(\C)$ dans $\F(\C)$ induise des isomorphismes entre groupes d'extensions. Le cas où $\C$ est additive a été complètement étudié : dès le degré cohomologique $i=3$, le morphisme canonique ${\rm Ext}^i_{\F_d(\C)}(F,G)\to {\rm Ext}^i_{\F(\C)}(F,G)$ peut cesser d'être un isomorphisme (même si $F$ et $G$ sont additifs et que l'on prend la colimite sur $d\in\mathbb{N}$). Dans les bons cas (avec une hypothèse de torsion bornée), ce morphisme est un isomorphisme lorsque $d$ est assez grand par rapport à $i$ et aux degrés de $F$ et $G$, mais pas en général. (Néanmoins, sans aucune hypothèse sur la catégorie additive $\C$, les inclusions $\F_d(\C;\mathbb{Q})\hookrightarrow\F(\C;\mathbb{Q})$ induisent des isomorphismes entre {\em tous} les groupes d'extensions.) Ces résultats sont établis par le deuxième auteur dans \cite{P-extor} dans un cas particulier ($F$ ou $G$ additif et pas de torsion dans les groupes abéliens de morphismes de $\C$) et par le premier auteur dans \cite{Dja-Pg} dans le cas général. Comme dans le présent travail, la démonstration de ces résultats nécessite d'analyser le gradué associé à la filtration de l'anneau d'un groupe (abélien, cette fois) par les puissances de l'idéal d'augmentation ; il faut toutefois aussi utiliser d'autres ingrédients, issus de la construction cubique de Mac Lane. Notons d'ailleurs que le c\oe ur de notre démonstration consiste à montrer que la restriction des foncteurs $\I^{d+1}$ aux groupes libres appartient à $\T_d$ ; l'argument central de \cite{P-extor} (sa proposition~4.3) consiste à montrer une propriété analogue pour la restriction des foncteurs $\I^{d+1}$ aux groupes {\em abéliens} libres, propriété qui n'est valide qu'en degré cohomologique assez petit (et qui repose sur un argument d'idéal {\em quasi-régulier}). 
\item Si $\Gamma$ désigne la catégorie des ensembles finis pointés, les inclusions $\F_d(\Gamma)\hookrightarrow\F(\Gamma)$ induisent des isomorphismes entre tous les groupes d'extensions. C'est une conséquence directe du théorème de type Dold-Kan établi par le deuxième auteur dans \cite{PDK}.
\item Notons $\mathbf{mon}$ (un squelette de) la catégorie des {\em monoïdes} libres de rang fini. Le théorème~\ref{thf} reste vrai si l'on remplace la catégorie source $\mathbf{gr}$ par $\mathbf{mon}$. Une méthode pour le voir consiste à reprendre les arguments du présent article et les adapter à la catégorie $\F(\mathbf{mon})$. Une autre consiste à utiliser le résultat dû à Hartl et aux deuxième et troisième auteurs (\cite{HPV}, {\em Corollary}~5.38) selon lequel le foncteur de complétion en groupe $\alpha : \mathbf{mon}\to\mathbf{gr}$ induit pour chaque $d\in\mathbb{N}$ une équivalence de catégories $\F_d(\mathbf{gr})\xrightarrow{\simeq}\F_d(\mathbf{mon})$. Dès lors, il suffit de voir que le foncteur $\alpha$ induit des isomorphismes entre groupes d'extensions
$${\rm Ext}^*_{\F(\mathbf{gr})}(F,G)\xrightarrow{\simeq} {\rm Ext}^*_{\F(\mathbf{mon})}(F\circ\alpha,G\circ\alpha)$$
lorsque $F$ et $G$ sont polynomiaux. C'est en fait vrai si l'on suppose seulement que $G$ est polynomial. Il suffit de le voir lorsque $F$ est un foncteur projectif $P^\mathbf{gr}_n$ ; on peut alors utiliser un argument classique reposant sur le fait que le morphisme canonique d'un monoïde {\em libre} vers sa complétion en groupe induit un isomorphisme en homologie. Cet argument est donné en détail, dans un contexte abélien analogue, dans \cite{Dja-JKT}, théorème~3.3.
 \end{enumerate}
\end{rem}

\section{Applications}\label{sa}

Notons $\mathfrak{a} : \mathbf{gr}\to\mathbf{Ab}$ le foncteur d'abélianisation. Ce foncteur joue un rôle fondamental dans la catégorie $\F(\mathbf{gr})$, où les calculs d'algèbre homologique sont grandement facilités par l'existence d'une résolution projective explicite de $\mathfrak{a}$. Celle-ci est donnée par la résolution bar (dont on tronque le degré nul : on utilise que l'homologie d'un groupe libre est naturellement isomorphe à son abélianisation en degré $1$ et nulle en degré $>1$), qui prend la forme :
\begin{equation}\label{eqb}
\dots\to P^\mathbf{gr}_{n+1}\to P^\mathbf{gr}_n\to\dots\to P^\mathbf{gr}_2\to P^\mathbf{gr}_1\to\mathfrak{a}\to 0.
\end{equation}
Cette résolution projective apparaît pour la première fois dans \cite{PJ} (§\,5.A) ; elle est également utilisée de façon fondamentale dans \cite{DV2}.

En utilisant la résolution bar {\em normalisée}, on obtient une variante de la précédente résolution, également utile :
\begin{equation}\label{eqbn}
\dots\to \bar{P}^{\otimes (n+1)}\to\bar{P}^{\otimes n}\to\dots\to\bar{P}^{\otimes 2}\to\bar{P}\to\mathfrak{a}\to 0.
\end{equation}

\begin{pr}\label{pr-dh}
 Soient $d\geq n>0$ des entiers. Le foncteur $\mathfrak{a}^{\otimes n}$ est de dimension homologique $d-n$ dans la catégorie $\F_d(\mathbf{gr})$ : on a ${\rm Ext}^i_{\F_d(\mathbf{gr})}(\mathfrak{a}^{\otimes n},-)=0$ pour $i>d-n$, tandis que le foncteur ${\rm Ext}^{d-n}_{\F_d(\mathbf{gr})}(\mathfrak{a}^{\otimes n},-)$ n'est pas nul.
\end{pr}

\begin{proof}
 Le foncteur ${\rm Ext}^i_{\F_d(\mathbf{gr})}(\mathfrak{a}^{\otimes n},-)$ est la restriction à $\F_d(\mathbf{gr})$ du foncteur ${\rm Ext}^i_{\F(\mathbf{gr})}(\mathfrak{a}^{\otimes n},-)$, d'après le théorème~\ref{thf}.

On suppose d'abord $n=1$. En utilisant la suite exacte~(\ref{eqbn}), on constate que ${\rm Ext}^i_{\F(\mathbf{gr})}(\mathfrak{a},F)$ s'identifie à l'homologie en degré $i$ d'un complexe de chaînes du type
$$\dots\to cr_{i+1}(F)(\mathbb{Z},\dots,\mathbb{Z})\to cr_i(F)(\mathbb{Z},\dots,\mathbb{Z})\to\dots\to cr_2(F)(\mathbb{Z},\mathbb{Z})\to cr_1(F)(\mathbb{Z})\,,$$
qui est nul à partir du degré $d$ si $F$ appartient à $\F_d(\mathbf{gr})$. Cela montre l'inégalité ${\rm hdim}_{\F_d(\mathbf{gr})}(\mathfrak{a})\leq d-1$.

On suppose maintenant $n>1$, et l'on établit l'inégalité ${\rm hdim}_{\F_d(\mathbf{gr})}(\mathfrak{a}^{\otimes n})\leq d-n$ (pour tout $d\geq n$) par récurrence sur $n$. En utilisant l'adjonction entre le coproduit catégorique $* : \mathbf{gr}\times\mathbf{gr}\to\mathbf{gr}$ et la diagonale, on obtient des isomorphismes naturels
$${\rm Ext}^*_{\F(\mathbf{gr})}(\mathfrak{a}^{\otimes n},F)\simeq {\rm Ext}^*_{\F(\mathbf{gr}\times\mathbf{gr})}(\mathfrak{a}^{\otimes (n-1)}\boxtimes\mathfrak{a},F\circ *),$$
où $\boxtimes$ désigne le produit tensoriel extérieur. Comme $\mathfrak{a}$ prend ses valeurs dans les groupes abéliens libres, on dispose pour tout bifoncteur $X$ sur $\mathbf{gr}$ d'une suite spectrale
$$E_2^{i,j}={\rm Ext}^i_{\F(\mathbf{gr})}\big(\mathfrak{a}^{\otimes (n-1)},G\mapsto {\rm Ext}^j_{\F(\mathbf{gr})}(\mathfrak{a},X(G,-))\big)\Rightarrow {\rm Ext}^{i+j}_{\F(\mathbf{gr}\times\mathbf{gr})}(\mathfrak{a}^{\otimes (n-1)}\boxtimes\mathfrak{a},X\circ *).$$
Le terme $E_2$ est inchangé si l'on remplace $X=F\circ *$ par son facteur direct $cr_2(F)$, car les foncteurs $\mathfrak{a}$ et $\mathfrak{a}^{\otimes (n-1)}$ sont réduits (puisque $n>1$). Mais le foncteur $G\mapsto {\rm Ext}^j_{\F(\mathbf{gr})}(\mathfrak{a},cr_2(F)(G,-))$ est de degré $d-1$ si $F$ est de degré $d$, de sorte que l'hypothèse de récurrence montre que
$${\rm Ext}^i_{\F(\mathbf{gr})}\big(\mathfrak{a}^{\otimes (n-1)},G\mapsto {\rm Ext}^j_{\F(\mathbf{gr})}(\mathfrak{a},X(G,-))\big)=0\;\text{si}\; i>(d-1)-(n-1)=d-n,$$
ce qui achève de prouver l'inégalité ${\rm hdim}_{\F_d(\mathbf{gr})}(\mathfrak{a}^{\otimes n})\leq d-n$.

L'inégalité ${\rm hdim}_{\F_d(\mathbf{gr})}(\mathfrak{a}^{\otimes n})\leq d-n$ se déduit de la non-nullité de ${\rm Ext}^{d-n}_{\F(\mathbf{gr})}(\mathfrak{a}^{\otimes n},\mathfrak{a}^{\otimes d})$, établie par le troisième auteur dans \cite{V-cal}.
\end{proof}

\begin{rem}\label{rq-fs}
L'analogue <<~abélien~>> de la proposition~\ref{pr-dh} n'est pas exact : notons $\mathbf{ab}$ la catégorie des groupes abéliens libres de rang fini. Dans la catégorie $\F_2(\mathbf{ab})$, le foncteur d'inclusion est de dimension homologique infinie et possède une résolution projective $4$-périodique, comme on le déduit de la section~6 de l'article \cite{P-extor}.

En revanche, on dispose quand même d'un résultat très similaire à la proposition~\ref{pr-dh} en remplaçant la catégorie source $\mathbf{gr}$ par une catégorie additive appropriée. Soient $p$ un nombre premier et $\F(p)$ la catégorie des foncteurs des $\mathbb{F}_p$-espaces vectoriels de dimension finie vers les $\mathbb{F}_p$-espaces vectoriels. Le foncteur d'inclusion est de dimension homologique finie $2p^{[{\rm log}_p(d)]}-2$ (où les crochets désignent la partie entière) dans la sous-catégorie $\F_d(p)$ des foncteurs polynomiaux de degré au plus $d$. Ce résultat est dû à Franjou et Smith (\cite{FrS}, §\,4.2).
\end{rem}

Les foncteurs $q_n(\bar{P})$ (appelés {\em foncteurs de Passi} dans \cite{HPV}) jouent un rôle fondamental dans la catégorie $\F(\mathbf{gr})$ (la démonstration du théorème~\ref{thf} en constitue une illustration) ; la proposition~\ref{pr-dh} permet facilement d'en calculer la dimension homologique.

\begin{cor}\label{cor-passi}
 Soient $d\geq n>0$ des entiers. Le foncteur $q_n(\bar{P})$ est de dimension homologique $d-n$ dans la catégorie $\F_d(\mathbf{gr})$.
\end{cor}

\begin{proof}
 On dispose de suites exactes courtes
$$0\to\mathfrak{a}^{\otimes n}\to q_n(\bar{P})\to q_{n-1}(\bar{P})\to 0$$
(cf. la proposition~\ref{ida} et le début de la section~\ref{sp}), d'où des inégalités
$${\rm hdim}_{\F_d(\mathbf{gr})}(q_{n-1}(\bar{P}))\leq\max\big(1+{\rm hdim}_{\F_d(\mathbf{gr})}(\mathfrak{a}^{\otimes n}),{\rm hdim}_{\F_d(\mathbf{gr})}(q_n(\bar{P}))\big).$$
On en déduit l'inégalité ${\rm hdim}_{\F_d(\mathbf{gr})}(q_n(\bar{P}))\leq d-n$ par récurrence descendante sur $n$, en utilisant la proposition~\ref{pr-dh} et le caractère projectif de $q_d(\bar{P})$ dans la catégorie $\F_d(\mathbf{gr})$.

L'inégalité inverse provient de ce que ${\rm Ext}^i_{\F(\mathbf{gr})}(q_n(\bar{P}),\mathfrak{a}^{\otimes (n+i)})$ est non nul (cf. \cite{V-cal}) et du théorème~\ref{thf}.
\end{proof}

Notre prochain résultat concerne les catégories $\F_d(\mathbf{gr};\mathbb{Q})$. On note que le théorème~\ref{thf} vaut aussi dans ce contexte (de même que pour $\F_d(\mathbf{gr};k)$ et $\F(\mathbf{gr};k)$, où $k$ est un anneau quelconque), en utilisant l'adjonction entre le foncteur de rationalisation $\mathbf{Ab}\to\mathbb{Q}\ti\mathbf{Mod}$ et l'oubli, qui se propage au but par postcomposition.

\begin{pr}\label{prdq}
 Soit $d>0$ un entier. La catégorie $\F_d(\mathbf{gr};\mathbb{Q})$ est de dimension globale $d-1$ : on a 
${\rm Ext}^i_{\F_d(\mathbf{gr};\mathbb{Q})}=0$
pour $i\geq d$, tandis que le foncteur ${\rm Ext}^{d-1}_{\F_d(\mathbf{gr};\mathbb{Q})}$ n'est pas nul.
\end{pr}

\begin{proof}
 On montre d'abord l'inégalité ${\rm gldim}\,\F_d(\mathbf{gr};\mathbb{Q})\leq d-1$. Par le théorème~\ref{thf}, cela équivaut à dire que la restriction à $\F_d(\mathbf{gr};\mathbb{Q})^{op}\times \F_d(\mathbf{gr};\mathbb{Q})$ du foncteur ${\rm Ext}^i_{\F(\mathbf{gr};\mathbb{Q})}$ est nulle pour $i\geq d$. Pour cela, on va montrer par récurrence sur $d\in\mathbb{N}^*$ l'assertion suivante :
$$\forall i\geq n\geq d\quad\forall F\in {\rm Ob}\,\F_d(\mathbf{gr};\mathbb{Q})\quad\forall G\in {\rm Ob}\,\F_n(\mathbf{gr};\mathbb{Q})\qquad {\rm Ext}^i(F,G)=0$$
(on omet dans la suite l'indice $\F(\mathbf{gr};\mathbb{Q})$ pour les groupes d'extensions).
 On suppose donc le résultat établi pour les entiers strictement inférieurs à $d$.

Les résultats de structure de la catégorie $\F_d(\mathbf{gr};\mathbb{Q})$ donnés dans \cite{DV2} (voir la section~2 et la démonstration de la proposition~5.3 de cet article) impliquent que, pour tout foncteur $F$ de $\F_d(\mathbf{gr};\mathbb{Q})$, il existe une représentation $M$ du groupe symétrique $\mathfrak{S}_d$ et un morphisme $u : F\to\mathfrak{a}^{\otimes d}\underset{\mathbb{Q}[\mathfrak{S}_d]}{\otimes} M$ dont le noyau $N$ et le conoyau $C$ appartiennent à $\F_{d-1}(\mathbf{gr};\mathbb{Q})$. Pour $d=1$, on peut même supposer que $u$ est un isomorphisme, puisque la partie constante des foncteurs se scinde. L'hypothèse de récurrence montre que ${\rm Ext}^i(N,G)$ et ${\rm Ext}^i(C,G)$ sont nuls pour $G$ dans $\F_n(\mathbf{gr};\mathbb{Q})$ et $i\geq n\geq d-1$. Par ailleurs, comme l'anneau $\mathbb{Q}[\mathfrak{S}_d]$ est semi-simple, le foncteur $T:=\mathfrak{a}^{\otimes d}\underset{\mathbb{Q}[\mathfrak{S}_d]}{\otimes} M$ est somme directe de facteurs directs de $\mathfrak{a}^{\otimes d}$. La proposition~\ref{pr-dh} montre donc que ${\rm Ext}^i(T,G)=0$ pour $G$ dans $\F_n(\mathbf{gr};\mathbb{Q})$ et $i>n-d\geq 0$. On en déduit ${\rm Ext}^i(F,G)=0$ pour $G$ dans $\F_n(\mathbf{gr};\mathbb{Q})$ et $i\geq n\geq d$, ce qui termine la démonstration de l'inégalité ${\rm gldim}\,\F_d(\mathbf{gr};\mathbb{Q})\leq d-1$.

L'inégalité ${\rm gldim}\,\F_d(\mathbf{gr};\mathbb{Q})\geq d-1$ se déduit de ce que le groupe abélien ${\rm Ext}^{d-1}_{\F(\mathbf{gr})}(\mathfrak{a},\mathfrak{a}^{\otimes d})$ est non seulement non nul, mais aussi sans torsion (cf. \cite{V-cal}).
\end{proof}

\begin{rem}\label{rq-dgl}
 Contrairement à ce qui advient pour la plupart des autres résultats du présent article, dont les analogues sur $\mathbf{ab}$ sont plus difficiles à montrer que sur $\mathbf{gr}$ (voir notamment la remarque~\ref{rq-fs}), quand ils ne sont pas faux, la si\-tu\-ation est plus simple pour $\F_d(\mathbf{ab};\mathbb{Q})$, qui est une catégorie {\em semi-simple} pour tout $d\in\mathbb{N}$ (et ce résultat classique est aisé à prouver).
\end{rem}

En revanche les catégories $\F_d(\mathbf{gr})$ (resp. $\F_d(\mathbf{gr};\mathbb{F}_p)$, où $p$ est un nombre premier) sont de dimension globale infinie dès que $d\geq 2$ (resp. $d\geq p$). Cela provient du même résultat pour l'algèbre de groupe $\mathbb{Z}[\mathfrak{S}_d]$ (resp. $\mathbb{F}_p[\mathfrak{S}_d]$) et de la proposition ci-dessous.

Rappelons tout d'abord (cf. \cite{DV2}) que l'on dispose, pour tout $d\in\mathbb{N}$, d'un foncteur ${\rm cr}_d : \F_d(\mathbf{gr})\to\mathbb{Z}[\mathfrak{S}_d]\ti\mathbf{Mod}$ associant l'effet croisé $cr_d(F)(\mathbb{Z},\dots,\mathbb{Z})$ à $F$ ; ce foncteur possède un adjoint à gauche $\alpha_d$ et un adjoint à droite $\beta_d$. Le foncteur $\alpha_d$ possède une expression explicite simple : $\alpha_d(M)=\mathfrak{a}^{\otimes d}\underset{\mathbb{Z}[\mathfrak{S}_d]}{\otimes} M$, mais il n'est pas exact (pour $d\geq 2$). En revanche, le foncteur $\beta_d$ ne semble pas posséder d'expression simple. Néanmoins, on a :
\begin{pr}\label{pr-beta}
 Pour tout $d\in\mathbb{N}$, le foncteur $\beta_d : \mathbb{Z}[\mathfrak{S}_d]\ti\mathbf{Mod}\to\F_d(\mathbf{gr})$ est exact. Il induit des isomorphismes naturels
$${\rm Ext}^*_{\F(\mathbf{gr})}(\beta_d(M),\beta_d(N))\simeq {\rm Ext}^*_{\F_d(\mathbf{gr})}(\beta_d(M),\beta_d(N))\simeq {\rm Ext}^*_{\mathbb{Z}[\mathfrak{S}_d]}(M,N).$$
\end{pr}
 
\begin{proof}
 Pour voir que $\beta_d$ est exact, il suffit de vérifier que son adjoint à gauche ${\rm cr}_d$ envoie les générateurs projectifs $q_d(P^\mathbf{gr}_n)$ de $\F_d(\mathbf{gr})$ (où $n$ parcourt $\mathbb{N}$) sur des $\mathbb{Z}[\mathfrak{S}_d]$-modules projectifs. Comme ${\rm cr}_d$ est exact et tue les foncteurs de degré strictement inférieur à $d$, il prend la même valeur sur $q_d(P^\mathbf{gr}_n)$ et le noyau $Q_n^d$ de la projection $q_d(P^\mathbf{gr}_n)\twoheadrightarrow q_{d-1}(P^\mathbf{gr}_n)$. Or on a $Q_1^d\simeq\mathfrak{a}^{\otimes d}$ (car le gradué de l'anneau d'un groupe libre est isomorphe à l'algèbre tensorielle de son abélianisation --- cf. section~\ref{sp} ; on utilise également la proposition~\ref{ida}). En général, en utilisant le premier point de la proposition~\ref{prgp} (et encore la proposition~\ref{ida}), on obtient :
$$Q_n^d\simeq\underset{i_1+\dots+i_n=d}{\bigoplus}\mathfrak{a}^{\otimes i_1}\otimes\dots\otimes\mathfrak{a}^{\otimes i_n}\simeq\underset{i_1+\dots+i_n=d}{\bigoplus}\mathfrak{a}^{\otimes d}.$$
Comme ${\rm cr}_d$ envoie le foncteur $\mathfrak{a}^{\otimes d}$ sur le $\mathbb{Z}[\mathfrak{S}_d]$-module $\mathbb{Z}[\mathfrak{S}_d]$, cela démontre l'exactitude de $\beta_d$. 

L'adjonction entre les foncteurs exacts $\beta_d$ et ${\rm cr}_d$ se propage aux groupes d'extensions :
$${\rm Ext}^*_{\F_d(\mathbf{gr})}(F,\beta_d(N))\simeq  {\rm Ext}^*_{\mathbb{Z}[\mathfrak{S}_d]}({\rm cr}_d (F),N).$$
Comme la coünité ${\rm cr}_d\beta_d\to {\rm Id}$ est un isomorphisme (cf. \cite{DV2}, §\,2, par exemple), en utilisant le théorème~\ref{thf}, on en déduit la dernière assertion de l'énoncé.
\end{proof}

\begin{rem}\begin{enumerate}
\item En reprenant la démonstration précédente, il est facile de donner une expression explicite de $\beta_d(M)(G)$ fonctorielle en $M$, {\em mais pas en $G\in {\rm Ob}\,\mathbf{gr}$}. 
\item La proposition~\ref{pr-beta} contraste encore avec la situation de $\F(\mathbf{ab})$ (ou plus généralement de $\F(\A)$, où $\A$ est une petite catégorie {\em additive}). On y dispose de même d'un foncteur
$$\F_d(\mathbf{ab})\to\mathbb{Z}[\mathfrak{S}_d]\ti\mathbf{Mod}\quad F\mapsto cr_d(F)(\mathbb{Z},\dots,\mathbb{Z})$$
qui possède des adjoints de chaque côté. L'adjoint à gauche est très similaire au foncteur $\alpha_d$ évoqué avant la proposition (il est donné par 
$M\mapsto T^d\underset{\mathbb{Z}[\mathfrak{S}_d]}{\otimes} M\simeq (T^d\otimes M)_{\mathfrak{S}_d}$, où $T^d$ désigne la $d$-ème puissance tensorielle). L'adjoint à droite est analogue à l'adjoint à gauche (il est donné par $M\mapsto (T^d\otimes M)^{\mathfrak{S}_d}$) ; c'est donc encore un foncteur explicite mais {\em non exact} si $d>1$. La référence originelle pour cette question classique est \cite{Pira-rec}.
           \end{enumerate}
\end{rem}

\bibliographystyle{plain}
\bibliography{bibli-DPV.bib}

\begin{thebibliography}{10}

\bibitem{Dja-Pg}
Aur\'elien Djament.
\newblock Groupes d'extensions et foncteurs polynomiaux.
\newblock Pr\'epublication disponible sur
  http://hal.archives-ouvertes.fr/hal-01023705.

\bibitem{Dja-JKT}
Aur\'elien Djament.
\newblock Sur l'homologie des groupes unitaires à coefficients polynomiaux.
\newblock {\em J. K-Theory}, 10(1):87--139, 2012.

\bibitem{DV2}
Aur\'elien Djament and Christine Vespa.
\newblock Sur l'homologie des groupes d'automorphismes des groupes libres \`a
  coefficients polynomiaux.
\newblock arXiv:1210.4030, \`a para\^itre \`a Commentarii Mathematici
  Helvetici.

\bibitem{EML}
Samuel Eilenberg and Saunders Mac~Lane.
\newblock On the groups {$H(\Pi,n)$}. {II}. {M}ethods of computation.
\newblock {\em Ann. of Math. (2)}, 60:49--139, 1954.

\bibitem{FFPS}
Vincent Franjou, Eric~M. Friedlander, Teimuraz Pirashvili, and Lionel Schwartz.
\newblock {\em Rational representations, the {S}teenrod algebra and functor
  homology}, volume~16 of {\em Panoramas et Synth\`eses [Panoramas and
  Syntheses]}.
\newblock Soci\'et\'e Math\'ematique de France, Paris, 2003.

\bibitem{FrS}
Vincent Franjou and Jeffrey~H. Smith.
\newblock A duality for polynomial functors.
\newblock {\em J. Pure Appl. Algebra}, 104(1):33--39, 1995.

\bibitem{HPV}
Manfred Hartl, Teimuraz Pirashvili, and Christine Vespa.
\newblock Polynomial functors from algebras over a set-operad and non-linear
  {M}ackey functors.
\newblock arXiv:1209.1607 ; \`a para\^itre \`a IMRN.

\bibitem{PJ}
Mamuka Jibladze and Teimuraz Pirashvili.
\newblock Cohomology of algebraic theories.
\newblock {\em J. Algebra}, 137(2):253--296, 1991.

\bibitem{Passi}
Inder Bir~S. Passi.
\newblock {\em Group rings and their augmentation ideals}, volume 715 of {\em
  Lecture Notes in Mathematics}.
\newblock Springer, Berlin, 1979.

\bibitem{P-add}
T.~I. Pirashvili.
\newblock Higher additivizations.
\newblock {\em Trudy Tbiliss. Mat. Inst. Razmadze Akad. Nauk Gruzin. SSR},
  91:44--54, 1988.

\bibitem{Pira-rec}
T.~I. Pirashvili.
\newblock Polynomial functors.
\newblock {\em Trudy Tbiliss. Mat. Inst. Razmadze Akad. Nauk Gruzin. SSR},
  91:55--66, 1988.

\bibitem{P-extor}
Teimuraz Pirashvili.
\newblock Polynomial approximation of {${\rm Ext}$} and {${\rm Tor}$} groups in
  functor categories.
\newblock {\em Comm. Algebra}, 21(5):1705--1719, 1993.

\bibitem{PDK}
Teimuraz Pirashvili.
\newblock Dold-{K}an type theorem for {$\Gamma$}-groups.
\newblock {\em Math. Ann.}, 318(2):277--298, 2000.

\bibitem{V-cal}
Christine Vespa.
\newblock Extensions between tensor powers composed by abelianisation.
\newblock En pr\'eparation.

\bibitem{Wei}
Charles~A. Weibel.
\newblock {\em An introduction to homological algebra}, volume~38 of {\em
  Cambridge Studies in Advanced Mathematics}.
\newblock Cambridge University Press, Cambridge, 1994.

\end{thebibliography}

\end{document}